\documentclass[12pt,reqno]{amsart}
\usepackage{amsmath,amssymb,amsfonts,amscd,latexsym,amsthm,mathrsfs,verbatim,comment,cite}
\usepackage[unicode]{hyperref}
\newtheorem*{mtheorem}{Main theorem}
\newtheorem{lemma}{Lemma}

\newcommand{\qbin}[2]{\genfrac{[}{]}{0pt}{}{#1}{#2}}
\renewcommand{\pmod}[1]{\;\allowbreak(\operatorname{mod}#1)}

\begin{document}

\title[\protect\lowercase{($q$-)}Supercongruences hit again]{($q$-)Supercongruences hit again}

\author{Wadim Zudilin}
\address{Department of Mathematics, IMAPP, Radboud University, PO Box 9010, 6500 GL Nijmegen, The Netherlands}
\email{w.zudilin@math.ru.nl}

\date{4 December 2020. \emph{Updated}: 3 February 2021.}

\dedicatory{In remembrance of Srinivasa Ramanujan,\\a $q$-number theorist and experimental mathematician}

\begin{abstract}
Using an intrinsic $q$-hypergeometric strategy, we generalise Dwork-type congruences
$H(p^{s+1})/H(p^s)\equiv H(p^s)/H(p^{s-1})\pmod{p^3}$ for $s=1,2,\dots$ and $p$ a prime, when $H(N)$ are truncated hypergeometric sums corresponding to the periods of rigid Calabi--Yau threefolds.
\end{abstract}


\maketitle

Number theory would not be full without $q$.
It would have missed partitions from combinatorics, cyclotomic polynomials from algebra, modular forms and $q$-hypergeometry from analysis\,---\,a lot of beautiful stuff whose existence we now take for granted.
There would not have been the phenomenon of Srinivasa Ramanujan, who creatively invented many $q$-objects and demonstrated the power of experiment in number theory.

I hope that the story in this note is very much in Ramanujan's spirit, though not particularly touching his discoveries.
While studying myself Ramanujan's famous formulae \cite{Ra14} for $1/\pi$, I was astonished to realise their reincarnation as supercongruences \cite{Zu09}.
It took some time for me to find out that, in many cases, both the former and latter originate from common $q$-analogues.
The principal body of that work \cite{GZ18,GZ19a,GZ19b,GZ20a,GZ20b} went in collaboration with Victor Guo and culminated in a simple analytical method of `creative microscoping' \cite{GZ19a,Zu20} that powerfully converts $q$-hypergeometric identities into congruences.
I must admit that Victor has done way more in this direction (see \cite{Gu19a,Gu19b,Gu20a,Gu20b,Gu20c,Gu20d,GS19,GS20a,GS20b,GS20c} and many other recent contributions by him) but there is still some room for further interesting results.
One of those `open slots' is our subject here.

\section{Supercongruences associated with modular Calabi--Yau threefolds}
\label{sec:CY}

In order to start the theme, we have to introduce some standard notation and conventions.
If you are already familiar with that, feel free to jump to the next paragraph.
We set $(a;q)_n=(1-a)(1-aq)\dotsb(1-aq^{n-1})$ to be the $q$-shifted factorial ($q$-Pochhammer symbol), with its multiple version
$$
(a_1,\dots,a_m;q)_k=\prod_{j=1}^m(a_j;q)_k.
$$
Further, let
\begin{align*}
\Phi_n(q)=\prod_{\substack{1\le k\le n\\ \gcd(n,k)=1}}(q-\zeta_n^k),
\end{align*}
be the $n$th cyclotomic polynomial, where $\zeta_n=e^{2\pi i/n}$ is an $n$th primitive root of unity.
Also recall the ordinary shifted factorial $(a)_n=\Gamma(a+n)/\Gamma(a)=a(a+\nobreak1)\allowbreak\dotsb(a+n-1)$ for $n=0,1,2,\dots$\,.
In what follows, the congruence $A_1(q)/A_2(q)\equiv0\pmod{P(q)}$
for polynomials $A_1(q),A_2(q),P(q)\in\mathbb Z[q]$ is understood as $P(q)$ divides $A_1(q)$ and is coprime with $A_2(q)$;
more generally, $A(q)\equiv B(q)\pmod{P(q)}$ for rational functions $A(q),B(q)\in\mathbb Z(q)$ means $A(q)-B(q)\equiv0\pmod{P(q)}$.

\medskip
Let $r_1,r_2\in\bigl\{\frac 12,\frac13,\frac14,\frac16\bigr\}$
or $(r_1,r_2)\in \bigl\{\bigl(\frac15,\frac25\bigr),\bigl(\frac18,\frac38\bigr), \bigl(\frac1{10},\frac3{10}\bigr), \bigl(\frac1{12},\frac5{12}\bigr)\bigr\}$ (there are precisely fourteen cases of those with $r_1\le r_2$) and $d=d_{r_1,r_2}$ denote the least common denominator of $r_1,r_2$.
For $N=1,2,\dots$, define the truncated $q$-hypergeometric sums
$$
H(N;q)=H_{r_1,r_2}(N;q)
=\sum_{k=0}^{N-1}\frac{(q^{dr_1},q^{d(1-r_1)},q^{dr_2},q^{d(1-r_2)};q^d)_k}{(q^d;q^d)_k^4}\,q^{dk}.
$$

\begin{mtheorem}
For positive integers $A,B,n$ with $(n,d)=1$ and $n>1$, the following $q$-supercongruence is true\textup:
\begin{equation}
\frac{H(An;q)}{H(Bn;q)}\equiv\frac{H(A;q^{n^2})}{H(B;q^{n^2})}\pmod{\Phi_n(q)^3}.
\label{eq:D}
\end{equation}
\end{mtheorem}

Notice that
$$
H(N)=\lim_{q\to1}H(N;q)
=\sum_{k=0}^{N-1}\frac{(r_1)_k(1-r_1)_k(r_2)_k(1-r_2)_k}{k!^4}
$$
and the above congruence (with the choices $A=p^s$, $B=p^{s-1}$ and $n=p$) implies
\begin{equation}
\frac{H(p^{s+1})}{H(p^s)}
\equiv\frac{H(p^s)}{H(p^{s-1})}\pmod{p^3}
\label{eq:C}
\end{equation}
provided that $H(p)\not\equiv0\pmod p$.
Dwork's theory in fact guarantees that
$$
\frac{H(p^{s+1})}{H(p^s)}
\equiv\frac{H(p^s)}{H(p^{s-1})}\pmod{p^s},
$$
so that the quotient $H(p^{s+1})/H(p^s)$ converges $p$-adically to a unit root $\gamma_p$ (see \cite[Theorem 1]{LTYZ17}).
The congruence \eqref{eq:C} valid for all $s=1,2,\dots$ implies that
$$
\gamma_p\equiv\frac{H(p^{s+1})}{H(p^s)}
\equiv\frac{H(p^1)}{H(p^{0})}\pmod{p^3}
=H(p),
$$
so that $H(p)\equiv\gamma_p\equiv a(p)\pmod{p^3}$, where $f=f_{r_1,r_2}=\sum_{n=1}^\infty a(n)q^n$ is the corresponding weight 4 modular form.
This comparison of truncated hypergeometric sum $H(p)$ with the $p$th coefficient of the modular form modulo $p^3$ was originally conjectured by Rodrigues-Villegas \cite{RV03} for the fourteen cases above; those correspond to the modularity instances of rigid Calabi--Yau threefolds whose periods are described through $_4F_3$ hypergeometric differential equations (refer to \cite{LTYZ17,RV03} for details of this setup).
The particular case $r_1=r_2=\frac12$ was conjectured earlier by Van Hamme \cite[Congruence~(M.2)]{VH97} and established later by Kilbourn~\cite{Ki06} (who was unaware of Van Hamme's paper).
The case $(r_1,r_2)=(\frac15,\frac25)$ of Rodrigues-Villegas' conjecture attached to the famous quintic threefold \cite{COGP91} was settled by McCarthy \cite{MC12}, 
while the reduction of case $(r_1,r_2)=(\frac14,\frac12)$ to Kilbourn's result from \cite{Ki06} was performed by McCarthy and Fuselier \cite{FM16}.
Finally, the uniform treatment of all fourteen cases was completed in our joint work \cite{LTYZ17} with Long, Tu and Yui,  using two(!) different methods.
The main theorem here leads to another proof of Rodriguez-Villegas' conjectures in~\cite{RV03}.

In fact, as observed by Roberts and Rodriguez-Villegas in \cite{RRV19} the expectation is that, for $s=1,2,\dots$,
\begin{equation*}
\frac{H(p^{s+1})}{H(p^s)}
\equiv\frac{H(p^s)}{H(p^{s-1})}\pmod{p^{3s}}
\end{equation*}
provided that $H(p)\not\equiv0\pmod p$. There seems to be, however, no clear way to $q$-cast this expectation for $s>1$.

Among all fourteen weight 4 modular forms corresponding to the rigid hypergeometric Calabi--Yau threefolds,
only $f_{1/4,1/3}=q\prod_{n=1}^\infty(1-q^{3n})^8$ is a CM modular form: we have $\gamma_p=-\Gamma_p\bigl(\frac 13\bigr)^9$ for $p\equiv1\pmod3$ and $a(p)=0$ for primes $p\equiv 2\pmod 3$. It is observed numerically in \cite{LTYZ17} that the supercongruence for this case also reflects the additional CM structure: for all primes $p\equiv 1\pmod 3$,
$$
\sum_{k=0}^{p-1}\frac{(\frac13)_k(\frac23)_k(\frac14)_k(\frac34)_k}{k!^4}
\equiv -\Gamma_p\Bigl(\frac 13\Bigr)^9 \pmod{p^4},
$$
which is sharp. 
(The latter $p$-adic gamma value is also featured in \cite[Congruence~(D.2)]{VH97}.
However its known $q$-analogues are obscurely complicated.)
In comparison,  the choice $p^3$ is sharp in all remaining thirteen cases.
This suggests that there might be a closed form for $H_{1/4,1/3}(n;q)\pmod{\Phi_n(q)^3}$ when $n\equiv1\pmod3$.

Finally, we notice the weaker companion to \eqref{eq:D} of shape
\begin{equation}
\frac{H(An;q)}{H(Bn;q)}\equiv\frac{H(A;q^n)}{H(B;q^n)}\pmod{\Phi_n(q)^2},
\label{eq:D0}
\end{equation}
which possesses a stronger form
\begin{equation}
\frac{H(An;q)}{H(Bn;q)}\equiv\frac{H(A;q^n)}{H(B;q^n)}+C(A,B)(n^2-1)(q^n-1)^2\pmod{\Phi_n(q)^3}
\label{eq:D1}
\end{equation}
for the quantity
\begin{equation}
C_{r_1,r_2}(A,B)=d^2\,\frac{H(A;1)\sum_{\ell=0}^{B-1}\ell(\ell+\mu)c(\ell;1)-H(B;1)\sum_{\ell=0}^{A-1}\ell(\ell+\mu)c(\ell;1)}{12\,H(B;1)^2}
\label{eq:CAB}
\end{equation}
with $\mu=\mu_{r_1,r_2}=r_1(1-r_1)+r_2(1-r_2)$.
Particular instances of \eqref{eq:D} and \eqref{eq:D0} are conjectured in \cite{Gu20d} in the case $r_1=r_2=\frac12$.
As we will see below, the congruences \eqref{eq:D1} are derivatives of the congruences \eqref{eq:D}, therefore they do not require separate attention.

There is a numerical evidence that the conditions on $r_1,r_2$ cannot be relaxed, so that the fourteen cases are the only ones when such supercongruences take place.

Denoting
\begin{equation}
c(k;q)=\frac{(q^{dr_1},q^{d(1-r_1)},q^{dr_2},q^{d(1-r_2)};q^d)_k}{(q^d;q^d)_k^4}\,q^{dk}=c(k;1/q)
\label{eq:cgen}
\end{equation}
the $k$th term of the sum $H(N;q)$ and writing \eqref{eq:D} in the form
\begin{equation*}
\frac{H(An;q)}{H(A;q^{n^2})}\equiv\frac{H(Bn;q)}{H(B;q^{n^2})}\pmod{\Phi_n(q)^3},
\end{equation*}
we see that the latter follows from
\begin{equation}
\sum_{k=0}^{n-1}\frac{c(\ell n+k;q)}{c(\ell n;q)}
\equiv\frac{c(\ell;q^{n^2})}{c(\ell n;q)}\sum_{k=0}^{n-1}c(k;q)\pmod{\Phi_n(q)^3}
\label{eq:c}
\end{equation}
to be true for all $\ell=1,2,\dots$\,. 
Also notice that there are `closed forms' for the prefactor on the right-hand side in \eqref{eq:c} in all the fourteen cases, thanks to \cite{St19} (see also \cite[Section~5]{Zu19}).

It does not seem realistic (and deserving!) to record detailed proofs for each case covered in the main theorem.
We illustrate the general strategy on the particular case $r_1=r_2=\frac12$, which is definitely lighter than the others and has a history of its own \cite{Gu20d,Ki06,VH97}.

\section{Details of proofs when $r_1=r_2=1/2$}
\label{sec:1/2}

Below we concentrate entirely on the case $r_1=r_2=\frac12$, so that
\begin{equation}
c(k;q)=\frac{(q;q^2)_k^4}{(q^2;q^2)_k^4}\,q^{2k}.
\label{eq:cf}
\end{equation}
In this situation we will often use the fact that $c(k;q)\equiv0$ modulo $\Phi_n(q)^4$ (hence modulo $\Phi_n(q)^3$ as well!) when the residue $k\pmod n$ exceeds $m=(n-1)/2$; in particular,
$$
\sum_{k=0}^{n-1}c(k;q)\equiv\sum_{k=0}^mc(k;q)\pmod{\Phi_n(q)^3}.
$$

\begin{lemma}
\label{lem:n^2}
In the case $r_1=r_2=\frac12$ we have, for $\ell=0,1,2,\dots$,
$$
\frac{c(\ell;q^{n^2})}{c(\ell n;q)}
\equiv c\Big(\frac{n-1}2;q\Big)^{2\ell}-\frac{\ell(2\ell+1)(n^2-1)}{6}(q^n-1)^2\pmod{\Phi_n(q)^3}.
$$
\end{lemma}

The resulting expression for $c(\ell;q^{n^2})/c(\ell n;q)$ makes sense, more generally, for $\ell\in\frac12\mathbb Z$.
Since the left-hand side in \eqref{eq:c} is also defined for those $\ell$, the congruence \eqref{eq:c} is implied by
\begin{equation}
Q(\ell;q)\equiv c\Big(\frac{n-1}2;q\Big)^{2\ell}-\frac{\ell(2\ell+1)(n^2-1)}{6}(q^n-1)^2\pmod{\Phi_n(q)^3},
\label{eq:cc}
\end{equation}
where
\begin{equation}
Q(\ell;q)=\sum_{k=0}^{(n-1)/2}\frac{c(\ell n+k;q)}{c(\ell n;q)}\Bigg/\sum_{k=0}^{(n-1)/2}c(k;q).
\label{eq:quot}
\end{equation}
In what follows we will be verifying the congruence \eqref{eq:cc}.

\begin{proof}[Proof of Lemma~\textup{\ref{lem:n^2}}]
Write
$$
c(k;q)=\frac{(q;q)_{2k}^4}{(q^2;q^2)_k^8}\,q^{2k}
={\qbin{2k}k}^4\frac{q^{2k}}{\prod_{j=1}^k(1+q^j)^8}
$$
and use \cite[Theorem~2.2]{St19} to get
\begin{align*}
\frac{c(\ell;q^{n^2})}{c(\ell n;q)}
&=\Bigg(\frac{{\qbin{2\ell}\ell}_{q^{n^2}}}{{\qbin{2\ell n}{\ell n}}}\Bigg)^4q^{2\ell n(n-1)}
\frac{\prod_{j=1}^{\ell n}(1+q^j)^8}{\prod_{j=1}^\ell(1+q^{jn^2})^8}
\\
&\equiv q^{2\ell n(n-1)}\frac{\prod_{j=1}^{\ell n}(1+q^j)^8}{\prod_{j=1}^\ell(1+q^{jn^2})^8}
+\frac{\ell^2(n^2-1)}{6}\,(q^n-1)^2\pmod{\Phi_n(q)^3}.
\end{align*}
Notice that for
$$
F_\ell(q)=\frac{\prod_{j=1}^{\ell n}(1+q^j)}{\prod_{j=1}^\ell(1+q^{jn^2})}
$$
we have $F_0(q)=1$ and
\begin{align*}
F_{\ell}(q)
&\equiv F_1(q)^\ell\cdot\bigg(1-\frac{n^2-1}8(q^n-1)^2\bigg)^{\ell(\ell-1)/2}
\\
&\equiv\biggl(\frac{\prod_{j=1}^n(1+q^j)}{1+q^{n^2}}\biggr)^\ell\bigg(1-\frac{\ell(\ell-1)(n^2-1)}{16}(q^n-1)^2\bigg)
\\
&\equiv\biggl(\frac{\prod_{j=1}^n(1+q^j)}{1+q^{n^2}}\biggr)^\ell-\frac{\ell(\ell-1)(n^2-1)}{16}(q^n-1)^2\pmod{\Phi_n(q)^3}
\end{align*}
for $\ell=0,1,\dots$, hence
$$
\frac{c(\ell;q^{n^2})}{c(\ell n;q)}
\equiv q^{2\ell n(n-1)}\biggl(\frac{\prod_{j=1}^n(1+q^j)}{1+q^{n^2}}\biggr)^{8\ell}-\frac{\ell(2\ell-3)(n^2-1)}{6}(q^n-1)^2\pmod{\Phi_n(q)^3}.
$$
Finally, notice the following congruence:
\begin{align}
\bigg(\frac{(q;q^2)_{(n-1)/2}}{(q^2;q^2)_{(n-1)/2}}\bigg)^2q^{(n-1)/2}
&\equiv q^{n(n-1)/2}\biggl(\frac{\prod_{j=1}^n(1+q^j)}{1+q^{n^2}}\biggr)^2
\nonumber\\ &\quad
+\frac{n^2-1}{6}(q^n-1)^2\pmod{\Phi_n(q)^3},
\label{eq:-1/2}
\end{align}
which follows from \cite[Theorem 1.2]{Pa07} (though stated there for $n$ odd primes but proved without the primality assumption; see also \cite[Eq.~(1.5)]{LPZ15}).
\end{proof}

\begin{lemma}
\label{lem:-1/2}
The congruence \eqref{eq:cc} is true for $\ell=-\frac12$.
\end{lemma}

\begin{proof}
We rearrange the summation:
\begin{align*}
\sum_{k=0}^{(n-1)/2}\frac{c(-\frac12n+k;q)}{c(-\frac12n;q)}
&=\sum_{k=0}^{(n-1)/2}\bigg(\frac{(q^{-n+1};q^2)_k}{(q^{-n+2};q^2)_k}\bigg)^4q^{2k}
\\
&=\bigg(\frac{(q^{-n+1};q^2)_{(n-1)/2}}{(q^{-n+2};q^2)_{(n-1)/2}}\bigg)^4
\sum_{k=0}^{(n-1)/2}\bigg(\frac{(q^{-1};q^{-2})_k}{(q^{-2};q^{-2})_k}\bigg)^4q^{n-1-2k}
\displaybreak[2]\\
&=\bigg(\frac{(q^2;q^2)_{(n-1)/2}}{(q;q^2)_{(n-1)/2}}\bigg)^4q^{-n+1}
\sum_{k=0}^{(n-1)/2}\bigg(\frac{(q;q^2)_k}{(q^2;q^2)_k}\bigg)^4q^{2k}
\\
&=c\Big(\frac{n-1}2;q\Big)^{-1}\sum_{k=0}^{(n-1)/2}c(k;q).
\end{align*}
This proves \eqref{eq:cc} for $\ell=-\frac12$.
\end{proof}

\begin{lemma}
\label{lem:Q}
For the quotient~\eqref{eq:quot} we have
\begin{align*}
Q(\ell;q)
&\equiv1-8\ell\Sigma_1(q)(q^n-1)
\\ &\quad
-4\ell(2\ell-1)\Sigma_1(q)(q^n-1)^2+8\ell^2\Sigma_2(q)(q^n-1)^2
\pmod{\Phi_n(q)^3},
\end{align*}
where
$$
\Sigma_1(q)=\frac{\sum_{k=0}^{(n-1)/2}c(k;q)S_1(k)}{\sum_{k=0}^{(n-1)/2}c(k;q)},
\quad
\Sigma_2(q)=\frac{\sum_{k=0}^{(n-1)/2}c(k;q)(4S_1(k)^2-S_2(k))}{\sum_{k=0}^{(n-1)/2}c(k;q)}
$$
and the `harmonic sums' $S_1(k)=S_1(k;q)$, $S_2(k)=S_2(k;q)$ are defined in~\eqref{eq:Snot} below.
\end{lemma}

\begin{proof}
We essentially apply to \eqref{eq:cf} the strategy from \cite{Zu19}. Namely, using
\begin{align*}
\frac{(aq;q)_k}{(q;q)_k}
&=1+(1-a)\sum_{j=1}^k\frac{q^j}{1-q^j}
\\ &\quad
+\frac{(1-a)^2}2\bigg(\bigg(\sum_{j=1}^k\frac{q^j}{1-q^j}\bigg)^2-\sum_{j=1}^k\frac{q^{2j}}{(1-q^j)^2}\bigg)+O\big((1-a)^3\big)
\end{align*}
with $a=q^{2\ell n}$, we get, for $0\le k\le(n-1)/2$,
\begin{align*}
\frac{c(\ell n+k;q)}{c(\ell n;q)c(k;q)}
&\equiv1+4S_1(k;q)(1-q^{2\ell n})
\\ &\quad
+2(4S_1(k;q)^2-S_2(k;q))(1-q^{2\ell n})^2
\pmod{\Phi_n(q)^3},
\end{align*}
where
\begin{equation}
\begin{aligned}
S_1(k)=S_1(k;q)
&=\sum_{j=1}^{2k}\frac{q^j}{1-q^j}-2\sum_{j=1}^k\frac{q^{2j}}{1-q^{2j}}
=\sum_{j=1}^k\frac{q^j}{1+q^j}+\sum_{j=k+1}^{2k}\frac{q^j}{1-q^j},
\\
S_2(k)=S_2(k;q)
&=\sum_{j=1}^{2k}\frac{q^{2j}}{(1-q^j)^2}-2\sum_{j=1}^k\frac{q^{4j}}{(1-q^{2j})^2}.
\end{aligned}
\label{eq:Snot}
\end{equation}
Then
\begin{align*}
\sum_{k=0}^{(n-1)/2}\frac{c(\ell n+k)}{c(\ell n)}
&\equiv\sum_{k=0}^{(n-1)/2}c(k)
+4(1-q^{2\ell n})\sum_{k=0}^{(n-1)/2}c(k)S_1(k)
\\ &\quad
+2(1-q^{2\ell n})^2\sum_{k=0}^{(n-1)/2}c(k)(4S_1(k)^2-S_2(k))
\\
&\equiv\sum_{k=0}^{(n-1)/2}c(k)
-4\big(2\ell(q^n-1)+\ell(2\ell-1)(q^n-1)^2\big)\sum_{k=0}^{(n-1)/2}c(k)S_1(k)
\\ &\quad
+8\ell^2(q^n-1)^2\sum_{k=0}^{(n-1)/2}c(k;q)(4S_1(k)^2-S_2(k))
\pmod{\Phi_n(q)^3}.
\qedhere
\end{align*}
\end{proof}

\begin{lemma}
\label{lem:eval}
For $m=(n-1)/2$ and the sums defined in \eqref{eq:Snot} we have
$S_1(m;q)+S_2(m;q)\equiv0\pmod{\Phi_n(q)}$.
\end{lemma}

\begin{proof}
Consider
$$
S(q)=S_1(m;q)+S_2(m;q)=\sum_{j=1}^{n-1}\frac{q^j}{(1-q^j)^2}-2\sum_{j=1}^{(n-1)/2}\frac{q^{2j}}{(1-q^{2j})^2}
$$
modulo $\Phi_n(q)$. Take $\zeta=\zeta_n$ to be an $n$th primitive root of unity. Then
$$
\sum_{j=1}^{(n-1)/2}\frac{\zeta^{2j}}{(1-\zeta^{2j})^2}
=\sum_{j=1}^{(n-1)/2}\frac{\zeta^{n-2j}}{(1-\zeta^{n-2j})^2}
=\frac12\sum_{j=1}^{n-1}\frac{\zeta^j}{(1-\zeta^j)^2}
$$
implying $S(\zeta)=0$. The equality precisely means that $S(q)\equiv0\pmod{\Phi_n(q)}$.
\end{proof}

\begin{lemma}
\label{lem:spec}
For $n$ odd, put $m=(n-1)/2$. Then
\begin{align*}
c\Big(\frac{n-1}2;q\Big)
&=\bigg(\frac{(q;q^2)_m}{(q^2;q^2)_m}\bigg)^4q^{2m}
\\
&\equiv1-2S_1(m)(q^n-1)
+(S_1(m)+2S_1(m)^2)(q^n-1)^2\pmod{\Phi_n(q)^3},
\end{align*}
with $S_1(m)=S_1(m;q)$ defined in~\eqref{eq:Snot}.
\end{lemma}

\begin{proof}
Consider
\begin{equation*}
T(q)=T(n;q)
=\bigg(\frac{(q;q^2)_m}{(q^2;q^2)_m}\bigg)^2q^m
=\frac{(q;q^2)_m^4}{(q;q)_{2m}^2}q^m
\end{equation*}
already featured in \eqref{eq:-1/2}. We get
\begin{align*}
T(q)
&=T(q^{-1})
=\frac{(q^{-1};q^{-2})_m^4}{(q^{-1};q^{-1})_{2m}^2}q^{-m}
=\frac{(q^{-n+2};q^2)_m^4}{(q^{-n+1};q)_{2m}^2}q^{-m}
\\
&\equiv\frac{(q^2;q^2)_m^4}{(q;q)_{2m}^2}q^{-m}\big(1-2(1-q^{-n})S_1(m)
+2(1-q^{-n})^2S_1(m)^2+(1-q^{-n})^2S_2(m)\big)
\\
&\equiv\frac{(q^2;q^2)_m^4}{(q;q)_{2m}^2}q^{-m}\big(1-2(1-q^{-n})S_1(m)
\\ &\quad\qquad
+2(1-q^{-n})^2S_1(m)^2-(1-q^{-n})^2S_1(m)\big)
\pmod{\Phi_n(q)^3}
\end{align*}
with the help of Lemma \ref{lem:eval}.
Comparing two expressions for $T(q)$ and using $1-q^{-n}\equiv(q^n-1)-(q^n-1)^2$ and $(1-q^{-n})^2\equiv(q^n-1)^2$ modulo $\Phi_n(q)^3$, we deduce the desired claim.
\end{proof}

\begin{lemma}
\label{lem:more}
In the previous notation, the following congruence is true\textup:
$$
\Sigma_1+4\Sigma_1^2-\Sigma_2\equiv\frac{n^2-1}{24}\pmod{\Phi_n(q)}.
$$
\end{lemma}

\begin{proof}
It follows from Lemma~\ref{lem:Q} that
\begin{align*}
Q(-\ell;q)Q(\ell;q)
&\equiv1-16\ell^2(\Sigma_1+4\Sigma_1^2-\Sigma_2)(q^n-1)^2
\\
&\equiv1-4(\Sigma_1+4\Sigma_1^2-\Sigma_2)(q^{2\ell n}-1)^2
\pmod{\Phi_n(q)^3},
\end{align*}
so that the required assertion is equivalent to
\begin{equation}
Q(-\ell;q)Q(\ell;q)
\equiv1-\frac{n^2-1}6(q^{2\ell n}-1)^2
\pmod{\Phi_n(q)^3}.
\label{eq:ver}
\end{equation}
Introduce the rational function
$$
F_n(a;q)=\sum_{k=0}^{n-1}\frac{(aq;q^2)_k^4}{(aq^2;q^2)_k^4}q^{2k}
$$
with the motive that $Q(\ell;q)=F_n(q^{2\ell n};q)/F_n(1;q)$. This means that if
$$
\frac{F_n(a;q)F_n(a^{-1};q)}{F_n(1;q)^2}
=1+f(q)(a-1)^2+O\big((a-1)^3\big)
$$
is the Taylor series expansion around $a=1$ (the coefficient of $a-1$ vanishes because of the symmetry $a\leftrightarrow1/a$ of the expression), then
$$
Q(-\ell;q)Q(\ell;q)
\equiv1+f(q)(q^{2\ell n}-1)^2
\pmod{\Phi_n(q)^3}.
$$
To compute $f(q)\pmod{\Phi_n(q)}$ we employ the following remarkable identity:
\begin{equation}
\frac{F_n(a;\zeta)F_n(a^{-1};\zeta)}{F_n(1;\zeta)^2}
=\bigg(\frac{na^{(n-1)/2}}{1+a+a^2+\dots+a^{n-1}}\bigg)^4
\label{eq:beau}
\end{equation}
valid for any $n$th primitive root of unity $\zeta$. It remains to observe that
$$
\frac{na^{(n-1)/2}}{1+a+a^2+\dots+a^{n-1}}
=1-\frac{n^2-1}{24}(a-1)^2-\frac{n^2-1}{24}(a-1)^3+O\big((a-1)^4\big),
$$
so that $f(\zeta)=-(n^2-1)/6$ implying $f(q)\equiv-(n^2-1)/6\pmod{\Phi_n(q)}$, hence the congruence \eqref{eq:ver} and our lemma.
\end{proof}

\begin{proof}[Proof of the main result]
Connecting the results of Lemmas \ref{lem:-1/2}, \ref{lem:Q} and then applying Lemma~\ref{lem:more} we conclude that
\begin{align*}
c\Big(\frac{n-1}2;q\Big)^{-1}
&\equiv1+4\Sigma_1(q^n-1)-(4\Sigma_1-2\Sigma_2)(q^n-1)^2
\\
&\equiv1+4\Sigma_1(q^n-1)-\bigg(2\Sigma_1-8\Sigma_1^2+\frac{n^2-1}{12}\bigg)(q^n-1)^2
\pmod{\Phi_n(q)^3}.
\end{align*}
On the other hand, from Lemma~\ref{lem:spec} we find out that
$$
c\Big(\frac{n-1}2;q\Big)^{-1}
\equiv1+2S_1(m)(q^n-1)-(S_1(m)-2S_1(m)^2)(q^n-1)^2\pmod{\Phi_n(q)^3}.
$$
The comparison of these two representations leads to
$$
S_1(m;q)\equiv2\Sigma_1(q)-\frac{n^2-1}{24}(q^n-1)\pmod{\Phi_n(q)^2}.
$$
Thus,
\begin{align*}
Q(\ell;q)
&\equiv1-4\ell\bigg(S_1(m)+\frac{n^2-1}{24}(q^n-1)\bigg)(q^n-1)-2\ell(2\ell-1)S_1(m)(q^n-1)^2
\\ &\quad
+4\ell^2\bigg(S_1(m)+2S_1(m)^2-\frac{n^2-1}{12}\bigg)(q^n-1)^2
\pmod{\Phi_n(q)^3}
\displaybreak[2]\\
&=1-4\ell S_1(m)(q^n-1)
\\ &\quad
+\bigg(2\ell S_1(m)+8\ell^2S_1(m)^2-\frac{\ell(2\ell+1)(n^2-1)}6\bigg)(q^n-1)^2
\\
&\equiv c\Big(\frac{n-1}2;q\Big)^{2\ell}-\frac{\ell(2\ell+1)(n^2-1)}6(q^n-1)^2
\pmod{\Phi_n(q)^3}
\end{align*}
establishing \eqref{eq:cc}, hence \eqref{eq:c} and \eqref{eq:D} for $r_1=r_2=\frac12$.
\end{proof}

\section{Concluding remarks}
\label{sec:final}

All the ingredients of our proof in Section~\ref{sec:1/2} extend more generally to the remaining thirteen cases $(r_1,r_2)$ of the main theorem.
The corresponding modifications of statements are not entirely trivial (and most are incredibly complicated!) but we do not think they deserve recording here.
It is however worth mentioning that the remarkable identity \eqref{eq:beau} is valid for \emph{any} rational function
$$
F_n(a;q)=F_{n;r_1,r_2}(a;q)
=\sum_{k=0}^{n-1}\frac{(aq^{dr_1},aq^{d(1-r_1)},aq^{dr_2},aq^{d(1-r_2)};q^d)_k}{(aq^d;q^d)_k^4}q^{dk}
$$
(allowing to deal one with \eqref{eq:cgen} in particular) and an $n$th primitive root of unity $\zeta$, whenever $(n,d)=1$.
Here $r_1,r_2$ are \emph{arbitrary} rationals on the interval $(0,1)$ and $d$ is their least common denominator.
Analogues of this result can be stated for other balanced hypergeometric sums.
One example (visually related to \cite[Conjecture~3.13]{GZ20b}) sources the rational function
$$
G_n(a;q)=G_{n;r}(a;q)
=\sum_{k=0}^{n-1}\frac{2(aq^{dr},aq^{d(1-r)};q^d)_k}{(aq^d;q^d)_k^2(1+q^{dk})}q^{dk},
$$
where $d$ is the denominator of rational $r\in(0,1)$ and $(n,2d)=1$; then
\begin{equation}
\frac{G_n(a;\zeta)G_n(a^{-1};\zeta)}{G_n(1;\zeta)^2}
=\bigg(n\,\frac{1-a+a^2-\dots-a^{n-2}+a^{n-1}}{1+a+a^2+\dots+a^{n-2}+a^{n-1}}\bigg)^2.
\label{eq:beau2}
\end{equation}
We leave proofs of identities \eqref{eq:beau} and \eqref{eq:beau2} as a homework to the intelligent reader.

\medskip
The form \eqref{eq:c} also suggests that $C(A,B)$ in \eqref{eq:D1} comes from the asymptotics of $c(\ell;q^{n^2})/c(\ell;q^n)$. In our $r_1=r_2=\frac12$ illustrative case we get
$$
\frac{c(\ell;q^{n^2})}{c(\ell;q^n)}
=1-\frac{\ell(2\ell+1)(n^2-1)}6(q^n-1)^2\pmod{\Phi_n(q)^3}
$$
implying
$$
c(\ell;q^{n^2})
=c(\ell;q^n)-\ell(2\ell+1)c(\ell;1)\cdot\frac{n^2-1}6(q^n-1)^2\pmod{\Phi_n(q)^3}
$$
and transforming \eqref{eq:c} into
\begin{align*}
&
\sum_{k=0}^{n-1}c(\ell n+k;q)
\equiv c(\ell;q^n)\sum_{k=0}^{n-1}c(k;q)
\\ &\qquad
-\ell(2\ell+1)c(\ell;1)\sum_{k=0}^{n-1}c(k;q)\cdot\frac{n^2-1}6(q^n-1)^2\pmod{\Phi_n(q)^3}
\end{align*}
for $\ell=1,2,\dots$\,.
This means that
$$
H(An;q)
\equiv \bigg(H(A;q^n)
-\sum_{\ell=0}^{A-1}\ell(2\ell+1)c(\ell;1)\cdot\frac{n^2-1}6(q^n-1)^2
\bigg)\sum_{k=0}^{n-1}c(k;q)\pmod{\Phi_n(q)^3}
$$
and leads to \eqref{eq:D1} with
$$
C(A,B)=\frac{H(A;1)\sum_{\ell=0}^{B-1}\ell(2\ell+1)c(\ell;1)-H(B;1)\sum_{\ell=0}^{A-1}\ell(2\ell+1)c(\ell;1)}{6H(B;1)^2}.
$$
This is precisely the form assumed in \eqref{eq:CAB} when $r_1=r_2=\frac12$,
and this generalises to the other thirteen cases along the lines.

\medskip
\noindent
\textbf{Acknowledgements.}
It is my pleasure to thank Bruce Berndt, Victor Guo, Michael Schlosser and Armin Straub for their valuable comments on this project.


\begin{thebibliography}{99}

\bibitem{COGP91}
\textsc{P. Candelas}, \textsc{X. de la Ossa}, \textsc{P.S. Green} and \textsc{L. Parkes},
A pair of Calabi--Yau manifolds as an exactly soluble superconformal theory,
\emph{Nuclear Phys. B} \textbf{359} (1991), no. 1, 21--74.

\bibitem{Gu19a}
\textsc{V.J.W. Guo},
Common $q$-analogues of some different supercongruences,
\emph{Results Math.} \textbf{74} (2019), Art.~131.

\bibitem{Gu19b}
\textsc{V.J.W. Guo},
Some $q$-congruences with parameters,
\emph{Acta Arith.} \textbf{190} (2019), 381--393.

\bibitem{Gu20a}
\textsc{V.J.W. Guo},
$q$-Analogues of two ``divergent'' Ramanujan-type supercongruences,
\emph{Ramanujan J.} \textbf{52} (2020), 605--624.

\bibitem{Gu20b}
\textsc{V.J.W. Guo},
Proof of a generalization of the (B.2) supercongruence of Van Hamme through a $q$-microscope,
\emph{Adv. Appl. Math.} \textbf{116} (2020), Art.~102016.

\bibitem{Gu20c}
\textsc{V.J.W. Guo},
$q$-Supercongruences modulo the fourth power of a cyclotomic polynomial via creative microscoping,
\emph{Adv. Appl. Math.} \textbf{120} (2020), Art.~102078.

\bibitem{Gu20d}
\textsc{V.J.W. Guo},
$q$-Analogues of Dwork-type supercongruences,
\emph{J. Math. Anal. Appl.} \textbf{487} (2020), Art.~124022.

\bibitem{GS19}
\textsc{V.J.W. Guo} and \textsc{M.J. Schlosser},
Some new $q$-congruences for truncated basic hypergeometric series,
\emph{Symmetry} \textbf{11} (2019), no.~2, Art.~268.

\bibitem{GS20a}
\textsc{V.J.W. Guo} and \textsc{M.J. Schlosser},
A new family of q-supercongruences modulo the fourth power of a cyclotomic polynomial,
\emph{Results Math.} \textbf{75} (2020), Art. 155.

\bibitem{GS20b}
\textsc{V.J.W. Guo} and \textsc{M.J. Schlosser},
A family of $q$-hypergeometric congruences modulo the fourth power of a cyclotomic polynomial,
\emph{Israel J. Math.} \textbf{240} (2020), 821--835.

\bibitem{GS20c}
\textsc{V.J.W. Guo} and \textsc{M.J. Schlosser},
Some $q$-supercongruences from transformation formulas for basic hypergeometric series,
\emph{Constr. Approx.} \textbf{53} (2021), 155--200.

\bibitem{GZ18}
\textsc{V.J.W. Guo} and \textsc{W. Zudilin},
Ramanujan-type formulae for $1/\pi$: $q$-analogues,
\emph{Integral Transforms Spec. Funct.} \textbf{29} (2018), no.~7, 505--513.

\bibitem{GZ19a}
\textsc{V.J.W. Guo} and \textsc{W. Zudilin},
A $q$-microscope for supercongruences,
\emph{Adv. Math.} \textbf{346} (2019), 329--358.

\bibitem{GZ19b}
\textsc{V.J.W. Guo} and \textsc{W. Zudilin},
On a $q$-deformation of modular forms,
\emph{J. Math. Anal. Appl.} \textbf{475} (2019), 1636--646.

\bibitem{GZ20a}
\textsc{V.J.W. Guo} and \textsc{W. Zudilin},
A common $q$-analogue of two supercongruences,
\emph{Results Math.} \textbf{75} (2020), Art.~46.

\bibitem{GZ20b}
\textsc{V.J.W. Guo} and \textsc{W. Zudilin},
Dwork-type supercongruences through a creative $q$-micro\-scope,
\emph{J. Combin. Theory Ser.~A} \textbf{178} (2021), Art.~105362.

\bibitem{FM16}
\textsc{J.G. Fuselier} and \textsc{D. McCarthy},
Hypergeometric type identities in the $p$-adic setting and modular forms,
\emph{Proc. Amer. Math. Soc.} \textbf{144} (2016), no. 4, 1493--1508.

\bibitem{Ki06}
\textsc{T. Kilbourn},
An extension of the Ap\'ery number supercongruence,
\emph{Acta Arith.} \textbf{123} (2006), no. 4, 335--348.

\bibitem{LPZ15}
\textsc{J. Liu}, \textsc{H. Pan} and \textsc{Y. Zhang},
A generalization of Morley's congruence,
\emph{Adv. Difference Equ.} (2015), Art.~2015:254.

\bibitem{LTYZ17}
\textsc{L. Long}, \textsc{F.-T. Tu}, \textsc{N. Yui} and \textsc{W. Zudilin},
Supercongruences for rigid hypergeometric Calabi--Yau threefolds,
\emph{Preprint} \href{http://arxiv.org/abs/1705.01663}{\texttt{arXiv:\,1705.01663 [math.NT]}} (2017).

\bibitem{MC12}
\textsc{D.~McCarthy}, 
On a supercongruence conjecture of Rodriguez-Villegas,
\emph{Proc. Amer. Math. Soc.} \textbf{140} (2012), no. 7, 2241--2254.

\bibitem{Pa07}
\textsc{H. Pan},
A $q$-analogue of Lehmer's congruences,
\emph{Acta Arith.} \textbf{128} (2007), no. 4, 303--318.

\bibitem{Ra14}
\textsc{S.~Ramanujan},
Modular equations and approximations to~$\pi$,
\emph{Quart. J. Math.} (\emph{Oxford}) \emph{Ser.}~(2) \textbf{45} (1914), 350--372.

\bibitem{RRV19}
\textsc{D.P. Roberts} and \textsc{F.~Rodriguez-Villegas},
Hypergeometric supercongruences,
in: 2017 \emph{MATRIX Annals}, MATRIX Book Ser. \textbf{2} (Springer, Cham, 2019), 435--439;
\emph{Preprint} \href{http://arxiv.org/abs/1803.10834}{\texttt{arXiv:\,1803.10834 [math.NT]}} (2018), 5~pp.

\bibitem{RV03}
\textsc{F.~Rodriguez-Villegas},
Hypergeometric families of Calabi--Yau manifolds,
in: \emph{Calabi--Yau varieties and mirror symmetry} (Toronto, ON, 2001), Fields Inst. Comm. \textbf{38} (Amer. Math. Soc., Providence, RI, 2003), 223--231.

\bibitem{St19}
\textsc{A. Straub},
Supercongruences for polynomial analogs of the Ap\'ery numbers,
\emph{Proc. Amer. Math. Soc.} \textbf{147} (2019), 1023--1036.

\bibitem{VH97}
\textsc{L. Van Hamme},
Some conjectures concerning partial sums of generalized hypergeometric series,
in: \emph{$p$-adic functional analysis} (Nijmegen, 1996),
Lecture Notes in Pure Appl. Math. \textbf{192} (Dekker, New York, 1997), 223--236.

\bibitem{Zu09}
\textsc{W.~Zudilin},
Ramanujan-type supercongruences,
\emph{J. Number Theory} \textbf{129} (2009), 1848--1857.

\bibitem{Zu19}
\textsc{W. Zudilin},
Congruences for $q$-binomial coefficients,
\emph{Ann. Combin.} \textbf{23} (2019), 1123--1135.

\bibitem{Zu20}
\textsc{W. Zudilin},
The method of creative microscoping,
\emph{RIMS K\^oky\^uroku} no.~2162 (Kyoto Univ., July 2020), 227--234;
\emph{Preprint} \href{http://arxiv.org/abs/1912.06829}{\texttt{arXiv:\,1912.06829 [math.NT]}}.

\end{thebibliography}
\end{document}